\documentclass[12pt,a4paper,reqno]{amsart}
\usepackage[utf8]{inputenc}
\usepackage[T1]{fontenc}
\usepackage{amsmath}
\usepackage{amsthm}
\usepackage{amssymb}
\usepackage[abbrev]{amsrefs}
\usepackage{mathrsfs}
\usepackage[dvipsnames]{xcolor}
\usepackage{bm}
\usepackage{bbm}
\usepackage{hyperref}
\usepackage{mathtools}
\usepackage[]{enumitem}
\newlist{FNenumerate}{enumerate}{1}
\setlist*[FNenumerate]{label=\textbf{Action\ \arabic*}, resume=FN}
\AtBeginDocument{\def\MR#1{}}
\makeatletter
\@namedef{subjclassname@2020}{%
\textup{2020} Mathematics Subject Classification}
\makeatother
\newtheorem{Claim}{}

\newtheorem{claim}[Claim]{Claim}
\newtheorem{thm}{}[section]
\newtheorem{theorem}[thm]{Theorem}
\newtheorem{corollary}[thm]{Corollary}
\newtheorem{lemma}[thm]{Lemma}

\newtheorem{proposition}[thm]{Proposition}

\theoremstyle{definition}

\numberwithin{equation}{section}
\allowdisplaybreaks
\newcommand{\littletaller}{\mathchoice{\vphantom{\big|}}{}{}{}}
\newcommand\restr[2]{{\left.\kern-\nulldelimiterspace #1 \littletaller \right|_{#2}}}
\newcommand{\Nnorm}[1]{{\left\vert\kern-0.25ex\left\vert\kern-0.25ex\left\vert #1\right\vert\kern-0.25ex\right\vert\kern-0.25ex\right\vert}}
\newcommand{\abs}[1]{\left\lvert#1\right\rvert}
\newcommand{\norm}[1]{\left\lVert#1\right\rVert}
\newcommand{\enbrace}[1]{\left\lbrace#1\right\rbrace}
\newcommand{\enpar}[1]{\left(#1\right)}

\newcommand{\enbrak}[1]{\left[#1\right]}
\newcommand{\Id}{ {\mathrm{Id}}}
\newcommand{\Fou}{ {\mathcal{F}}}
\newcommand{\St}{ {\mathcal{S}}}

\newcommand{\Gt}{ {\mathcal{G}}}

\newcommand{\Dt}{ {\mathcal{D}}}
\newcommand{\EB}{ {\mathcal{E}}}
\newcommand{\YB}{ {\mathcal{Y}}}
\newcommand{\XB}{ {\mathcal{X}}}
\newcommand{\It}{ {\mathcal{I}}}
\newcommand{\Ot}{ {\mathcal{O}}}
\newcommand{\Pt}{ {\mathcal{P}}}

\newcommand{\Nt}{ {\mathcal{N}}}

\newcommand{\Mt}{ {\mathcal{M}}}
\newcommand{\Sym}{ {\mathbb{S}}}
\newcommand{\FF}{ {\mathbb{F}}}

\newcommand{\XX}{ {\mathbb{X}}}
\newcommand{\YY}{ {\mathbb{Y}}}

\newcommand{\NN}{ {\mathbb{N}}}

\newcommand{\cs}{ {\bm{s}}}

\newcommand{\dd}{ {\bm{d}}}
\newcommand{\ee}{ {\bm{e}}}
\newcommand{\xx}{ {\bm{x}}}
\newcommand{\yy}{ {\bm{y}}}

\DeclareMathOperator{\spn}{span}

\DeclareMathOperator\supp{supp}
\newcommand{\renorm}[1]{{\left\vert\kern-0.25ex\left\vert\kern-0.25ex\left\vert #1 
    \right\vert\kern-0.25ex\right\vert\kern-0.25ex\right\vert}}
\begin{document}
\title[Symmetry-type conditions of conditional bases]{Remarks about symmetry-type conditions of conditional bases of Banach spaces}
\author[Ansorena]{Jos\'e L. Ansorena}
\address{Department of Mathematics and Computer Sciences\\
Universidad de La Rioja\\
Logro\~no\\
26004 Spain}
\email{joseluis.ansorena@unirioja.es}
\author[Marcos]{Alejandro Marcos}
\address{Department of Mathematics and Computer Sciences\\
Universidad de La Rioja\\
Logro\~no\\
26004 Spain}
\email{alejandro.marcos@unirioja.es}
\begin{abstract}
We investigate the existence of equivalent $p$-norms, $0\le p \le 1$, under which conditional symmetric or spreading bases in quasi-Banach spaces become isometric. For spreading bases (which need not be unconditional or even Schauder bases), we develop new techniques involving the geometry of spreading sequences and their associated spreading models. We prove that any spreading basis is automatically seminormalized, M-bounded, and uniformly spreading, which allows the construction of an isometric renorming via its spreading model. For symmetric bases, we show they are necessarily spreading and uniformly symmetric, enabling a direct application of a renorming lemma for uniformly bounded semigroups of operators. Consequently, any quasi-Banach space with a symmetric basis admits a renorming making all permutations isometries, and any spreading basis admits a renorming making all increasing maps isometries. These results extend and unify classical isometric renorming theorems for unconditional, subsymmetric, and symmetric Schauder bases to the conditional, non-Schauder setting.
\end{abstract}
\thanks{J.\@ L.\@ Ansorena acknowledges the support of the Spanish Ministry for Science and Innovation under Grant PID2022-138342NB-I00 for \emph{Functional Analysis Techniques in Approximation Theory and Applications (TAFPAA)}}
\thanks{}
\maketitle
\section{Introduction}\noindent
The existence of `bases' of a certain type is one of the elements that determine the structure of separable Banach or, more generally, quasi-Banach spaces. Before we  discuss the types of bases studied in this paper, we point out that the minimum requirement we must impose on a sequence $\XB=(\xx_n)_{n=1}^\infty$ in a quasi-Banach space $\XX$ (over the real or complex field $\FF$) to call it a basis is completeness, along with the existence of a sequence $\XB^*=(\xx_n^*)_{n=1}^\infty$ in the dual space $\XX^*$ such that 
\begin{equation}\label{eq:BO}
\xx_n^*(\xx_k)=\delta_{n,k}, \quad n,k\in\NN,
\end{equation}
where $\delta_{n,k}$ is Kronecker's delta function defined by $\delta_{n,k}=1$ if $n=k$ and $\delta_{n,k}=0$ otherwise. A sequence $\XB$ for which there exists such a sequence $\XB^*$ is called a \emph{minimal system}. We say that $\XB^*$, which is unique provided that $\XB$ is complete, is a sequence of \emph{coordinate functionals} for $\XB$. A sequence $(\xx_n,\xx_n^*)_{n=1}^\infty$ in $\XX\times \XX^*$ that fulfils \eqref{eq:BO} is a \emph{biorthogonal system}.

A biorthogonal system $(\xx_n,\xx_n^*)_{n=1}^\infty$ yields for each $f\in\XX$ a formal series
\[
\sum_{n=1}^\infty \xx_n^*(f) \, \xx_n.
\]
This series might not determine $f$ even when $\XB=(\xx_n)_{n=1}^\infty$ is complete, that is, there are complete minimal systems $\XB$ such that the \emph{coefficient transform} relative to $\XB$ 
\[
f\mapsto \Fou[\XB](f)\coloneqq\enpar{\xx_n^*(f)}_{n=1}^\infty, \quad f\in\XX,
\]
is not one-to-one (see, e.g., \cite{HMVZ2008}). The systems in which this pathology does not occur are called \emph{total}.  In other words, a complete minimal system is total if and only if the weak* closed linear span of its coordinate functionals is the entire space $\XX^*$. Complete total minimal systems are called \emph{Markushevich bases}.

Let $B_\XX$ and $S_\XX$ denote, respectively, the closed unit ball and the unit sphere. Let $c_{00}$ denote the space of all eventually null scalar-valued sequences. Let $\Pt_{<\infty}(\Nt)$ (resp., $\Pt_\infty(\Nt)$) be the set of all finite (resp., infinite) subsets of a given set $\Nt$. Given a complete minimal system $\XB=(\xx_n)_{n=1}^\infty$ with coordinate functionals $\XB^*=(\xx_n^*)_{n=1}^\infty$, we say that $\lambda=(\lambda_n)_{n=1}^\infty$ is a \emph{multiplier} relative to $\XB$ if the linear map
\[
 f\mapsto \sum_{n\in \NN} \lambda_n \, \xx^*_n(f)\, \xx_n, \quad f\in\spn(\XB).
\]
extends to a bounded operator
\[
M_\lambda=M_\lambda[\XB,\XX]\colon \XX \to \XX.
\]
We say that $A\in\Pt(\NN)$ is a \emph{character} if $\chi_A$ is a multiplier, in which case we denote by
\[
S_A=S_A[\XB,\XX]=M_{\chi_A}
\]
the corresponding \emph{coordinate projection}. Any $\lambda\in c_{00}$ is a multiplier, and any $A\in \Pt_{<\infty}(\NN)$ is a character.

Many properties of complete minimal systems can be defined through these operators. Let us gather some of them.
\begin{enumerate}[label=(\alph*),leftmargin=*]
    \item If $S_{\enbrace{n}}$, $n\in\NN$, is uniformly bounded, we say that $\XB$ is an \emph{M-bounded} basis. Note that 
    \[
    \norm{S_{\enbrace{n}}}= \norm{\xx_n} \norm{\xx_n^*}, \quad n\in\NN.
    \]
    \item\label{it:Schauder} A sequence $(\xx_n)_{n=1}^\infty$ in $\XX$ such that any $f\in\XX$ can by univocally expanded as 
    \[
    f=\sum_{n=1}^\infty a_n \, \xx_n,
    \]
    where $a_n\in\FF$ for all $n\in\NN$, and the series is norm-convergent, is called a \emph{Schauder basis}. Let $\It$ be the set of all integer intervals in $\Pt_{<\infty}(\NN)$. It is known that $\XB$ is a Schauder basis if and only if it is a complete minimal system such that $S_I$, $I\in\It$, is uniformly bounded. Any Schauder basis is an M-bounded Markushevich basis.
    
    \item A sequence as in \ref{it:Schauder} satisfying the more demanding condition that the series converges unconditionally is called an \emph{unconditional basis}. It is known that the following are equivalent.
    \begin{itemize}
        \item $\XB$ is an unconditional basis.
        \item  $S_A$, $A\in\Pt_{<\infty}(\NN)$, is uniformly bounded.
        \item $M_\lambda$, $\lambda\in c_{00}\cap B_{\ell_\infty}$, is uniformly bounded.
    \end{itemize}
     If $\XB$ is an unconditional basis, then any sequence in $\ell_\infty$ is a multiplier, and $M_\lambda$, $\lambda\in B_{\ell_\infty}$, is uniformly bounded. 
\end{enumerate}

We refer the reader to \cite{AABW2021} for the basics on bases within the general framework of quasi-Banach spaces. 

As there are quasi-Banach spaces with no nontrivial functional \cite{Day1940}, the existence of bases is not guaranteed in general. Any separable Banach space has a Markusevich basis \cite{Markusevich1943}, and we can find it M-bounded \cites{Pel1976, Plic77}. However, there are separable Banach spaces with no Schauder basis \cite{Enflo1973}. Also, there are Banach spaces with a Schauder basis that fail to have an unconditional basis. This is the case of $L_1$ \cite{LinPel1968}, or the James space from \cite{James1951}. 

Other important properties of bases concern symmetry. To define them, given sets $\Mt$ and $\Nt$, we denote by $\Gamma(\Mt,\Nt)$ 
the set of all one-to-one maps from $\Mt$ to $\Nt$. Let $\XB=(\xx_n)_{n=1}^\infty$ be a complete minimal system with coordinate functionals $(\xx_n^*)_{n=1}^\infty$. We consider for each $\Nt\in\Pt(\NN)$ and $\psi\in\Gamma(\Nt,\NN)$ the  linear map
\begin{equation}\label{eq:shift}
 f\mapsto \sum_{n\in \Nt} 
\xx^*_n(f) \, \xx_{\psi(n)}, \quad f\in\spn(\XB).
\end{equation}
If this mapping \eqref{eq:shift} extends to a bounded operator
\[
L_\psi=L_\psi[\XB,\XX]\colon\XX\to \XX,
\]
we say that $\psi$ is a \emph{shift} relative to $\XB$, and that  $L_\psi$ is a \emph{shifting operator}. Note that if $\varphi$ and $\varphi$ are shifts, then so is $\psi\circ \varphi$, and $L_{\psi\circ\varphi}=L_\psi \circ L_\varphi$.

Let $\Pi(\Nt)$ be the set of all permutations of a set $\Nt$. If every map in $\Pi(\NN)$ is a shift relative to $\XB$, we say that $\XB$ is a \emph{symmetric basis}. It is well-known that symmetric Schauder bases are unconditional. However, this classical result does not extend to general complete minimal systems. In fact, there are Banach spaces with symmetric complete minimal systems that fail to be Markushevich bases (see \cite{AABCO2024}*{Example 3.9}). If $\XB$ is a symmetric unconditional basis, then $M_\psi$, $\psi\in\Pi(\NN)$, is a uniformly bounded family of operators.

Given partially ordered sets $\Mt$ and $\Nt$, we denote by $\Upsilon(\Mt,\Nt)$  the set of all increasing maps from $\Mt$ to $\Nt$. A complete minimal system is said to be a \emph{spreading basis} if any $\psi\in\Upsilon(\NN,\NN)$ is a shift, and the associated shifting operator $L_\psi$ is an isomorphism. Any symmetric Schauder basis is spreading (see \cite{Singer1961}). The only features of symmetric Schauder bases that one needs in many situations are spreadingness and unconditionality.  In fact, it was believed that any \emph{subsymmetric}, that is, spreading and unconditional, basis was symmetric until Garling \cite{Garling1968} proved otherwise.

M-bounded bases, Schauder bases, unconditional bases, symmetric Schauder bases, and subsymmetric bases can be characterized in terms of the uniform boundedness of suitable families of operators. In light of this, it is natural to ask whether suitable renormings make these properties hold isometrically, that is, with the norms of the associated operators bounded by one. As a matter of fact, the answer is positive in all cases and, in hindsight, we can provide proofs that follow the same pattern. Namely, we can infer the existence of the wished-for renormings from the following result, which we state in the general framework of quasi-Banach spaces. Recall that a quasi-Banach is said to be \emph{locally $p$-convex}, $0<p\le 1$, if the origin has a $p$-convex neighbourhood. By the Aoki--Rolewich theorem \cites{Aoki1942,Rolewicz1957}, any quasi-Banach space is locally $p$-convex for some $p$, in which case it can be equipped with an equivalent $p$-norm.

\begin{lemma}\label{lem:SG}
Let $0<p\le 1$ and $\XX$ be a locally $p$-convex quasi-Banach space. Suppose that $\Gt$ is a uniformly bounded semigroup of endomorphisms of $\XX$. Then, there is an equivalent $p$-norm for $\XX$ relative to which $\norm{T}\le 1$ for all $T\in \Gt$.
\end{lemma}

\begin{proof}
Assume without loss of generality that $\Id_\XX\in\Gt$. If $\norm{\cdot}$ is a $p$-norm for $\XX$, then the mapping
\[
f\mapsto \sup_{T\in\Gt} \norm{T(f)}, \quad f\in\XX,
\]
defines an equivalent $p$-norm, and for any $f\in\XX$ and $S\in\Gt$ we have
\[
\Nnorm{S(f)}=\sup_{T\in  \Gt \circ S} \norm{T(f)} \le \Nnorm{f}.\qedhere
\]
\end{proof} 

If $\XB$ is M-bounded, then Lemma~\ref{lem:SG} applies with 
\[
\Gt=\Gt_r\coloneqq\enbrace{0} \cup \enbrace{ S_{\enbrace{n}} \colon n\in\NN}.
\]
If $\XB$ is a Schauder basis, then Lemma~\ref{lem:SG} applies with
\[
\Gt=\Gt_b\coloneqq\enbrace{ S_I \colon I\in \It}.
\]
In turn, if $\XB$ is an unconditional basis, $\XB$ applies with
\[
\Gt=\Gt_u\coloneqq\enbrace{M_\lambda \colon \lambda\in B_{\ell_\infty}}.
\]
If $\XB$ is a symmetric Schauder basis, then $\XB$ applies with 
\[
\Gt=\Gt_m\coloneqq\enbrace{L_\psi\colon \psi\in \Pi(\NN)},
\]
so we get a renorming relative to which $L_\psi$ is an isometry for every $\psi\in\Pi(\NN)$. If we use
\[
\Gt_{u,m}=\enbrace{M_\lambda \circ T\colon T\in\Gt_m, \, \lambda\in B_{\ell_\infty}}
\]
instead, the renorming we get has the additional property that $M_\lambda$ is a contraction for every $\lambda\in B_{\ell_\infty}$.

The semigroup associated with subsymmetric bases is somewhat more obscure. Namely, $\XB$ is subsymmetric if and only if the semigroup  
\[
\Gt_s=\bigcup_{A\in\Pt_{<\infty}(\NN)}\enbrace{L_\psi\colon \psi\in \Upsilon(A,\NN)}
\]
is uniformly bounded. Hence, if we feed Lemma~\ref{lem:SG} with $\Gt_s$ we obtain a renorming relative to which $L_\psi$ is an isometric embedding for every $\psi\in\Upsilon(\NN,\NN)$. If we feed  Lemma~\ref{lem:SG} with the larger semigroup
\[
\Gt_{u,s}=\enbrace{M_\lambda \circ T \colon T\in \Gt_s, \, \lambda\in B_{\ell_\infty}}
\]
instead, then, $M_\lambda$ turns into a contraction for every $\lambda\in B_{\ell_\infty}$ (see \cite{Ansorena2018}). 

We say that a complete minimal system is \emph{conditional} if it fails to be unconditional. In this paper, we seek isometric renormings of quasi-Banach spaces with symmetric or spreading conditional bases. In the symmetric case, the semigroup $\Gt_m$ is still available, so we must only care about the uniform boundedness of the operators involved. In contrast, it is by no means clear whether a uniformly bounded semigroup of operators related to spreadingness exists. To achieve isometric renormings of quasi-Banach spaces with conditional spreading bases, we must pursue different techniques. 

We conclude this introductory section by setting some additional terminology we will use. Given a set $A$ in a quasi-Banach space $\XX$, we denote by $\enbrak{A}$ its closed linear span. A sequence $\XB=(\xx_n)_{n=1}^\infty$ in $\XX$ is said to be \emph{uniformly separated} if 
\[
\inf\enbrace{ \norm{\xx_n-\xx_k} \colon k,n\in\NN, \, k\not=n}>0.
\]
We say that $\XB$ is \emph{semi-normalized} if 
\[
\inf_{n\in\NN} \norm{\xx_n}>0, \quad \sup_{n\in\NN}\norm{\xx_n}<\infty.
\]
Suppose that $\XB=(\xx_n)_{n=1}^\infty$ is a complete minimal system  with coordinate functionals $\XB^*=(\xx_n^*)_{n=1}^\infty$. Then, $\XB$ is semi-normalized and M-bounded if and only if
\[
\sup_{n\in\NN} \max\enbrace{\norm{\xx_n}, \norm{\xx_n^*}}<\infty.
\]
(see \cite{AABW2021}*{Lemma 2.5}). The boundedness of $\XB^*$ implies that the coefficient transform $\Fou[\XB]$ is a bounded operator from $\XX$ into $c_0$.

The \emph{support} of $f\in\XX$ relative to $\XB$ is the set 
\[
\supp(f)=\enbrace{n\in\NN \colon \xx_n^*(f)\not=0}.
\]
If $f\in\spn(\XB)$, then $\supp(f)$ is finite and, regardless whether $\XB$ is total or not,
\[
f= S_{\supp(f)}(f).
\]

We record for further reference the elementary fact that any map in 
\[
\Ot\coloneqq\bigcup_{\Nt\in\Pt_\infty(\Nt)} \Upsilon(\Nt,\NN)
\]
is univocally determined by its domain and its range. Bearing this in mind,  given $\Nt\in\Pt_\infty(\NN)$, it is useful to denote by $\psi_\Nt$ the map in $\Upsilon(\NN,\NN)$ defined by  $\psi(\NN)=\Nt$.
\begin{lemma}\label{lem:DRI}
Given $\Nt$, $\Mt\in\Pt_\infty(\NN)$, there is a unique  $\psi\in\Upsilon(\Nt,\NN)$ such that $\psi(\Nt)=\Mt$. In fact,
$\psi=\psi_\Mt \circ \psi_\Nt^{-1}$.
\end{lemma}
\section{Renorming quasi-Banach spaces with spreading bases}\label{sect:Spreading}\noindent
We start our study of conditional spreading bases with a series of lemmas that work for sequences that may not be minimal systems. To state them, given a sequence $(\xx_n)_{n=1}^\infty$ in a quasi-Banach space $\XX$ and $\Nt\in\Pt(\NN)$, we say that $\psi\in\Gamma(\Nt,\NN)$ is \emph{translation} relative to $\XB$ if the linear map
\[
\spn\enpar{\xx_n \colon n\in\Nt} \to \XX, \quad \xx_n\mapsto \xx_{\psi(n)}
\]
extends to an isomorphism 
\[
T_\psi\colon \enbrak{\xx_n \colon n\in\Nt} \to  \enbrak{\xx_n \colon n\in \psi(\Nt)}.
\]
Note that if $\varphi$ and $\psi$ are translations, then so are $\psi\circ\varphi$ and $\psi^{-1}$.
Furthermore, since we are not assuming unconditionality, the \emph{translation operator} $T_\psi$ may not extend to a shifting operator $L_\psi$ even when $\XB$ is a complete minimal system. We say that $\XB$ is a \emph{spreading sequence} if every $\psi\in\Upsilon(\NN,\NN)$ is a translation. If $\XB$ is a spreading basis, then $T_\psi=L_\psi$ and $T_{\psi^{-1}}=L_\psi^{-1}$ for all $\psi\in\Upsilon(\NN,\NN)$. So, our definition is consistent.

By Lemma~\ref{lem:DRI}, any map in $\Ot$ is a translation relative to any spreading sequence. A spreading sequence for which there is a constant $C\in[1,\infty)$ such that $\norm{T_\psi}\le C$ for all  $\psi\in\Ot$ is said to be \emph{$C$-spreading}. A $1$-spreading sequence will be called \emph{isometrically spreading}. Needless to say, constant sequences are isometrically spreading.

\begin{lemma}\label{lem:SpreadingDiference}
Let $\XB=(\xx_n)_{n=1}^\infty$ be a spreading sequence in a quasi-Banach space $\XX$.
\begin{enumerate}[label=(\roman*)]
    \item\label{it:SubSeq} $(\xx_{\psi(n)})_{n=1}^\infty$ is spreading for all $\psi\in\Upsilon(\NN,\NN)$.
    \item\label{it:Diff} $\YB=\enpar{\xx_{2n}-\xx_{2n-1}}_{n=1}^\infty$ is spreading. 
\end{enumerate}
\end{lemma}
\begin{proof}
It is clear that \ref{it:SubSeq} holds. To prove \ref{it:Diff}, set $\yy_n=\xx_{2n}-\xx_{2n-1}$ for all $n\in\NN$. Fix $\psi\in\Upsilon(\NN,\NN)$. Define $\varphi\colon\NN\to\NN$ by $\varphi(2n-1)=2\psi(n)-1$ and    $\varphi(2n)=2\psi(n)$ for all $n\in\NN$. We have that $\varphi\in\Upsilon(\NN,\NN)$, so it is a translation relative to $\XB$. Since $T_\varphi(\yy_n)=\yy_{\psi(n)}$ for all $n\in\NN$, the restriction of $T_\varphi$ to $\enbrak{\yy_n, n\in\NN}$ witnesses that $\psi$ is a translation relative to $\YB$.
\end{proof}

\begin{lemma}\label{lem:SNSubSym}
Let $\XB=(\xx_n)_{n=1}^\infty$ be a spreading sequence in a quasi-Banach space $\XX$. If $\XB$ is linearly dependent, then it is constant.
\end{lemma}

\begin{proof}
Suppose that $\XB$ is linearly dependent. Pick $A\in \Pt_{<\infty}(\NN)\setminus\{\emptyset\}$ and $(a_n)_{n\in A}\in(\FF\setminus\{0\})^A$ such that 
\[
f\coloneqq \sum_{n\in A} a_n \, \xx_n=0.
\]
Set $k=\max(A)$. Given $n\in\NN\cap[k+1,\infty)$, there is $\psi\in\Upsilon(\NN,\NN)$ such that $\psi(k)=n$ and $\psi(j)=j$ whenever $j<k$. We have
\[
a_k\enpar{\xx_{n}-\xx_k}=T_{\psi}(f)-f=0.
\]
Hence, $\xx_n=\xx_k$.  Now, given $m\in\NN\cap[1,k-1]$, there is $\varphi\in\Upsilon(\NN,\NN)$ such that $\psi(m)=k$.  Since $\varphi(k)>k$,
\[
T_\varphi(\xx_m)=\xx_k=\xx_{\varphi(k)}=T_\varphi(\xx_k).
\]
Hence $\xx_m=\xx_k$.
\end{proof}

\begin{lemma}\label{lem:SNSubSymTris}
Let $\XB=(\xx_n)_{n=1}^\infty$ be a spreading sequence in a quasi-Banach space $\XX$. If $\XB$ is not constant, then it is uniformly separated.
\end{lemma}

\begin{proof}
Assume that $\XB$ has a convergent subsequence. Pick $A\in\Pt_\infty(\NN)$ such that $(\xx_n)_{n\in A}$ converges. Let $\varphi\in\Upsilon(A,\NN)$ be onto. We have 
\[
\lim_{n\in A} \xx_{\varphi(n)}=\lim_{n\in A} T_\varphi(\xx_n)=g\coloneqq L_\varphi\enpar{ \lim_{n\in A} \xx_n}.
\]
Consequently, $\lim_{n\in\NN} \xx_n=g$. Then, for any $\psi\in\Upsilon(\NN,\NN)$,
\[
T_\psi(g)=\lim_{n\in\NN} T_\psi(\xx_n)=\lim_{n\in\NN} \xx_{\psi(n)}=g.
\]
Since, by Lemma~\ref{lem:SNSubSym}, $\xx_n\not=\xx_k$ whenever $n\not=k$, $\xx_n\not=g$ for all $n\in\NN$. Therefore, regardless of whether $\XB$ has or does not has a convergent subsequence, no $\xx_n$, $n\in\NN$, is a limit of any subsequence of $\XB$. We infer that
\begin{equation}\label{eq:Separation}
\inf\enbrace{\norm{\xx_n-\xx_k} \colon n\in A, \, k\in\NN, \, k\not=n }>0, \quad A\in\Pt_{<\infty}(\NN).   
\end{equation}
Assume by contradiction that $\XB$ is not uniformly separated. Fix $(\varepsilon_j)_{j=1}^\infty$ in $(0,\infty)$ with $\lim_j \varepsilon_j=0$.  Use \eqref{eq:Separation} to recursively construct $\rho\in\Upsilon(\NN,\NN)$ such that the sequence $\YB=(\yy_j)_{j=1}^\infty$ given by
\[
\yy_j = \xx_{\rho(2j)}- \xx_{\rho(2j-1)}
\]
satisfies $\norm{\yy_{j+1}}\le \varepsilon_j \norm{\yy_j}$. On the one hand the mapping $j\mapsto j+1$ is not a translation relative to $\YB$. On the other hand, $\YB$ is spreading by Lemma~\ref{lem:SpreadingDiference}. This absurdity ends the proof.
\end{proof}

\begin{lemma}\label{lem:SNSubSymBis}
Let $\XB=(\xx_n)_{n=1}^\infty$ be a spreading sequence in a quasi-Banach space $\XX$. Then, $\XB$ is bounded.
\end{lemma}

\begin{proof}
By Lemma~\ref{lem:SNSubSym}, we can suppose that $\XB$ is linearly independent. Assume by contradiction that $\XB$ is unbounded. Pick $(R_j)_{j=1}^\infty$ in $(0,\infty)$ unbounded. Use that $\xx_n\not=0$ for all $n\in\NN$ to recursively construct $\psi\in\Upsilon(\NN,\NN)$ such that
\[
\norm{\xx_{\psi(j+1)}} \ge R_j \norm{\xx_{\psi(j)}}, \quad j\in\NN.
\]
This implies that the mapping $j\mapsto j+1$ is not translation relative to $\XB$ and we are done.
\end{proof}

Given $k\in\NN$, let $\rho_k\in\Upsilon(\NN,\NN)$ be defined by $\rho_k(\NN)=\NN\setminus\{k\}$. Note that if $(\xx_n)_{n=1}^\infty$ is spreading, then 
\begin{equation}\label{eq:Dif}
\enpar{T_{\rho_k}-T_{\rho_{k+1}}}(\xx_n)=\delta_{k,n}\enpar{\xx_{k+1}-\xx_k}, \quad k,n\in\NN.
\end{equation}

\begin{proposition}\label{prop:spreadbasic}
Let $\XB=(\xx_n)_{n=1}^\infty$ be a spreading sequence in a quasi-Banach space $\XX$. If $\XB$ is not constant, then it is a minimal system of its closed linear span. Besides, if $(\xx_n^*)_{n=1}^\infty$ are the coordinate functionals of $\XB$,
\[
\norm{\xx_{n+1}-\xx_n} \norm{\xx_n^*} = \norm{T_{\rho_n}-T_{\rho_{n+1}}}, \quad n\in\NN.
\]
\end{proposition}

\begin{proof} 
Assume with loss of generality that $\enbrak{\XB}=\XX$. Fix $k\in\NN$. By \eqref{eq:Dif}, there is $\xx_k^*\in\XX^*$ such that 
\[
\enpar{T_{\rho_k}-T_{\rho_{k+1}}}(f)=\xx_k^*(f) \enpar{\xx_{k+1}-\xx_k}, \quad f\in\XX.
\]
Since $\xx_k\not=\xx_{k+1}$ by Lemma~\ref{lem:SNSubSym},  $\xx_k^*(\xx_n)=\delta_{k,n}$ for all $n\in\NN$.
\end{proof}

Proposition~\ref{prop:spreadbasic} exhibits a dichotomy regarding spreading sequences. Namely, they are either constant sequences or spreading bases. So, from now on, we state our results for spreading bases.

\begin{proposition}\label{prop:uniform-spread}
 Let $\XB=(\xx_n)_{n=1}^\infty$ be a spreading basis of a quasi-Banach space $\XX$. Then, there is a constant $C\in[1,\infty)$ such that $\XB$ is $C$-spreading.
\end{proposition}

\begin{proof}
Let $(\xx_n^*)_{n=1}^\infty$ be the coordinate functionals of  $\XB$. Assume that $\XX$ is a $p$-Banach space, $0<p\le 1$. Use Lemma~\ref{lem:SNSubSymBis} to pick $K\in[1,\infty)$ such that
\[
\norm{\xx_n}\le K, \quad n\in\NN.
\]

Assume by contradiction that there is no $C\in[1,\infty)$ such that $\XB$ $C$-spreading.  Then, either
\begin{enumerate}[label=(\Alph*)]
    \item\label{lem:uniform-spread:A} 
    $\sup\enbrace{ \norm{L_\psi} \colon\psi\in\Upsilon(\NN,\NN)}=\infty,$ or
    \item\label{lem:uniform-spread:B} 
    $\sup\enbrace{ \norm{L_\psi^{-1}} \colon\psi\in\Upsilon(\NN,\NN)}=\infty$.
\end{enumerate}
Pick an unbounded sequence $(R_k)_{k=1}^\infty$ in $(1,\infty)$. We recursively construct $(f_k,\psi_k)_{k=1}^\infty$ in $\spn(\XB)\times \Upsilon(\NN,\NN)$ as follows. Set $f_0=0$, and let $\psi_0$ be the identity map. Fix $k\in\NN$ and assume that $(f_{k-1},\psi_{k-1})$ are constructed. Set, with the convention that $\max(\emptyset)=0$,
\begin{align*}
m_k&=\max\enpar{\supp(f_{k-1})}, \quad
M_k=\psi_{k-1}(m_k),\\
A_k&=\NN\cap[1,M_k],\quad D_k=\enpar{\sum_{n\in A_k} \norm{\xx_n^*}^p}^{1/p}.
\end{align*}
Then, choose $f_k\in\spn(\XB)$ and $\psi_k\in\Upsilon(\NN,\NN)$ such that
\[
\norm{f_k'}=D_k, \quad \norm{f_k''} > K D_k R_k ,
\]
where $f_k'=f_k$ and $f_k''=L_{\psi_k}(f_k)$ if \ref{lem:uniform-spread:A} holds, and $f_k'=L_{\psi_k}(f_k)$ and $f_k''=f_k$ if \ref{lem:uniform-spread:B} holds. 

Set $g_k=f_k-S_{A_k}(f_k)$ for all $k\in\NN$. Put $g_k'=g_k$  and $g_k''=L_{\psi_k}(g_k)$ if \ref{lem:uniform-spread:A} holds, and  $g_k'=L_{\psi_k}(g_k)$ and $g_k''=g_k$ if \ref{lem:uniform-spread:B} holds. We have
\[
\norm{g_k'}^p \le D_k^p + K^p D_k^p, \quad \norm{ g_k''}^p> K^p R_k^p D_k^p-K^p D_k^p.
\]
In particular, $g_k\not=0$. We have
\begin{align*}
 1+m_k\le 1+M_k \le &\min(\supp(g_k)), \\
 &\max(\supp(g_k))= \max(\supp(f_k))=m_{k+1}.
\end{align*}
Set $\Nt_k= [1+m_k, m_{k+1}]\cap \NN$ for all $n\in\NN$. If we define $\Nt= \cup_{k=1}^\infty \Nt_k$ and $\psi\colon \Nt\to \NN$ by 
\[
\restr{\psi}{\Nt_k}=\restr{\psi_k}{\Nt_k}, \quad k\in\NN,
\]
then $\psi\in\Upsilon(\Nt,\NN)$. 

On the one hand, $T_\psi$ is an isomorphic embedding. On the other hand, since $T_\psi(g_k)=L_{\psi_k}(g_k)$,
\[
\max \enbrace{ \frac{\norm{T_\psi(g_k)}}{\norm{g_k}}, \frac{\norm{g_k}}{\norm{T_\psi(g_k)}} }\ge  \frac{K}{\enpar{1+K^p}^{1/p}} \enpar{R_k^p-1}^{1/p}.
\]
This absurdity ends the proof. 
\end{proof}

\begin{corollary}\label{cor:SSMBounded}
Any spreading basis of any quasi-Banach space is semi-normalized and M-bounded.   
\end{corollary}

\begin{proof}
Let $\XX$ be a quasi-Banach space. Let $\XB$ be spreading basis of $\XX$ with coordinate functionals $\XB^*$. Combining Proposition~\ref{prop:spreadbasic}, Lemma~\ref{lem:SNSubSymTris} and Proposition~\ref{prop:uniform-spread}, gives that $\XB^*$ is bounded. In turn, $\XB$ is bounded by Lemma~\ref{lem:SNSubSymBis}.
\end{proof}

The renormings of quasi-Banach spaces with spreading bases constructed in this paper rely on spreading models. Let us gather some basics on this topic. For each $N\in\NN$, we endow the set
\[
\enbrak{\NN}^{(N)}\coloneqq \Upsilon ( \NN\cap[1,N] , \NN)
\]
with the partial ordering given by $\alpha\le\beta$ if $\alpha(1)\le\beta(1)$. Note that $[\NN]^{(N)}$ is a directed set. Following \cite{BL1983}, we say that a sequence $\XB=(f_n)_{n=1}^\infty$ in a quasi-Banach space $\XX$ is \emph{good} if for all $N\in\NN$ and $\alpha\coloneqq(a_n)_{n=1}^N\in\FF^N$ there exists
\[
F_N(\alpha)\coloneqq \lim_{\varphi\in[\NN]^{(N)}}\norm{\sum_{j=1}^N a_j \, \xx_{\varphi(j)}} <\infty.
\]
Any good sequence is bounded. If $M< N$, $\alpha=(a_j)_{j=1}^M\in\FF^M$ and $\beta=(b_j)_{j=1}^N\in\FF^N$ are such $a_j=b_j$ for all $j=1$, \dots, $M$, and $b_j=0$ for all $j=M+1$, \dots,$N$, then $F_M(\alpha)=F_N(\beta)$. Hence, we can safely define
\[
\norm{\cdot}_\XB\colon c_{00} \to [0,\infty), \quad (a_n)_{n=1}^\infty \mapsto F_N\enpar{(a_n)_{n=1}^N},
\]
where $N\in\NN$ is such that $a_n=0$ for all $n\in\NN\cap[N+1,\infty)$. The map $\norm{\cdot}_\XB$ is a semi-quasi-norm on $ c_{00}$. If $\XX$ is a $p$-Banach space for some $p\in[0.1)$, then $\norm{\cdot}_\XB$ a semi-$p$-norm.

Let $\EB=(\ee_n)_{n=1}^\infty$ be the unit vector system of $\FF^\NN$ and, for each $n\in\NN$, let $\ee_n^*\colon\FF^\NN\to\FF$ be the $n$th coordinate operator. Given $\psi\in\Upsilon(\NN,\NN)$, let $T_\psi\colon\FF^\NN\to\FF^\NN$ be the shiting operator relative to $\EB$. We have $\norm{f}_\XB=\norm{L_\psi(f)}_\XB$ for all $f\in c_{00}$ and $\psi\in\Upsilon(\NN,\NN)$. Set 
\[
\nu[\XB] \coloneqq \norm{\ee_1-\ee_2}_\XB= \lim_{(j,k)\in[\NN]^{(2)}}\norm{e_j-e_k}.
\]
We have 
\[
\nu[\XB] \abs{\ee_n^*(f)} \le \norm{L_{\rho_n}-L_{\rho_{n+1}}}  \norm{f}_\XB, \quad f\in c_{00},
\]
(see Proposition~\ref{prop:spreadbasic}). Therefore, if $\nu[\XB]>0$ then $\norm{\cdot}_\XB$ is a quasi-norm, and $\EB$ is an isometrically spreading basis of the completion of $(c_{00}, \norm{\cdot}_\XB)$. We say that this completion, denoted by $\Sym[\XB]$, is the \emph{spreading model} associated with $\XB$. It is known that, given a good sequence $\XB$ with $\nu[\XB]>0$ in a Banach space $\XX$, then $\EB$ is an unconditional, whence subsymmetric, basis of $\Sym[\XB]$ if and only if $\XB$ is weakly null (see \cite{BL1983}). To the best of our knowledge, unconditional spreading models in locally nonconvex spaces are by now poorly understood.

By the Infinite Ramsey Theorem, every bounded sequence $\XB$ has a good subsequence $\YB$ (see \cite{BL1983}); besides, $\nu[\YB]>0$ provided that $\XB$ is uniformly separated. 

\begin{theorem}
Let $\XB$ be a spreading basis of a locally $p$-convex quasi-Banach space $\XX$, $0<p\le 1$. Then there is an equivalent $p$-norm for $\XX$ that makes $\XB$ isometrically spreading.
\end{theorem}
\begin{proof}
We can assume that $\XX$ is equipped with a $p$-norm. Passing to a subsequence, we can also assume that $\XB=(\xx_n)_{n=1}^\infty$ is good. Since $\XB$ is uniformly separated by Lemma~\ref{lem:SNSubSymTris}, $\nu[\XB]>0$. Hence, it suffices to prove that $\XB$ is equivalent to the canonical basis of $\Sym[\XB]$. Let $T\colon\spn(\XB)\to c_{00}$ be the linear map defined by $\xx_n\mapsto \ee_n$, $n\in\NN$. Define for each $m\in\NN$ $\beta_m\in\Upsilon(\NN,\NN)$ by 
\[
\beta_m(n)=m+n, \quad n\in\NN.
\]
The mere definition of $\norm{\cdot}_\XB$ gives, for all $f\in\spn(\XB)$,
\[
N_l(f)\coloneqq \liminf_m \norm{L_{\beta_m}(f)}=N_u(f)\coloneqq \limsup_m \norm{L_{\beta_m}(f)}=\norm{T(f)}_\XB.
\]
 Since Proposition~\ref{prop:uniform-spread} yields a constant $C$ such that
\[
\frac{1}{C} N_l(f) \le \norm{f}\le C N_u(f), \quad f\in\spn(\XB),
\]
we are done.
\end{proof}

We close this section by pointing out that conditional spreading Schauder bases do exist. A classic example arises within the study of convergent series. Let endow the linear space
\[
\cs=\enbrace{(a_n)_{n=1}^\infty \in\FF^\NN \colon \sum_{n=1}^\infty a_n \mbox{ converges} }
\]
with the norm
\[
\norm{(a_n)_{n=1}^\infty}=\max_{k\in\NN} \abs{\sum_{n=k}^\infty a_n}.
\]
Then, the \emph{unit vector basis} $\EB$  is an isometrically spreading Schauder basis of $\cs$. The obvious isomorphism
\[
T\colon \cs \to c_0, \quad (a_n)_{n=1}^\infty \mapsto \enpar{\sum_{n=k}^\infty a_n}_{k=1}^\infty
\]
transforms $\EB$ into the \emph{summing basis} 
\[
\enpar{\sum_{k=1}^n \ee_k}_{n=1}^\infty.
\]
The sequence of coordinate functionals of the summning basis of $c_0$ is the \emph{difference system} $\Dt=(\dd_n)_{n=1}^\infty$ of $\ell_1$ given by
\[
\dd_n=\ee_n-\ee_{n+1}, \quad n\in\NN.
\]
Note that $\Dt$ is weak*-closed in $\ell_1$, but fails to be norm-closed. Also note that $\Dt$ is not spreading. In fact, while $(\dd_{2n})_{n=1}^\infty$ is equivalent to the canonical $\ell_1$-basis, $\norm{\sum_{n=1}^m \dd_n}_1=2$ for all $m\in\NN$.

In contrast,  the coordinate functionals $\XB^*$ of a symmetric basis $\XB$ form a symmetric basis of $\enbrak{\XB^*}$. Indeed, given a shift $\psi\colon\Nt \subset\NN \to \NN$ relative to a complete minimal system $\XB=(\xx_n)_{n=1}^\infty$ with coordinate functionals $\XB^*=(\xx_n^*)_{n=1}^\infty$, the dual operator of the corresponding shifting operator  satisfies 
\[
\restr{L_\psi^*[\XB,\XX]}{\enbrak{\XB^*}}=L_{\psi^{-1}}[\XB^*,\enbrak{\XB^*}].
\]
Hence, $\XB^*$ is symmetric (resp., subsymmetric) provided that $\XB$ is symmetric (resp., subsymmetric). 

The summing basis of $c_0$ suggests a connection between the unconditionality of a spreading basis and the spreadingness of its dual basis. The following result substantiates this connection.

\begin{proposition}
Let $\XB$ be a Schauder spreading basis of a Banach space. Then $\XB^*$ is spreading if and only if $\XB$ is unconditional. 
\end{proposition}

\begin{proof}
If $\XB$ is unconditional, then it is subsymmetric. Hence, $\XB^*$ is subsymmetric. In particular, $\XB^*$ is spreading. 

Assume that $\XB^*$ is spreading. Then, $\psi^{-1}$ is shift relative to $\XB^{**}$ for all $\psi\in\Upsilon(\NN,\NN)$. Since $\XB^{**}$ is equivalent to $\XB$ (see \cite{AlbiacKalton2016}*{Corollary 3.2.4.}), we infer from Lemma~\ref{lem:DRI} that any map in $\Ot$ is a shift relative to $\XB$, and $(L_\varphi)_{\varphi\in\Ot}$ is a uniformly bounded family of operators. In particular, any set in  $\Pt_\infty(\NN)$ is a character relative to $\XB$, and 
$(S_\Nt)_{\Nt\in\Pt_\infty(\NN)}$ is a uniformly bounded family of operators. Consequently,
\[
(S_A)_{A\in\Pt_{<\infty}(\NN)}=(\Id_\XX -S_\Nt)_{\Nt\in\Pt_\infty(\NN)}
\]
is also a uniformly bounded family of operators.
\end{proof}
\section{Renorming quasi-Banach spaces with symmetric bases}\noindent
Similarly to spreading sequences, we say that a sequence  $\XB=(\xx_n)_{n=1}^\infty$ in a quasi-Banach space $\XX$ is \emph{symmetric} if every $\psi\in\Pi(\NN)$ is a translation relative to $\XB$. If $C\in[1,\infty)$ is such that $\norm{L_\psi}\le C$ for all $\psi\in\Pi(\NN)$, $\XB$ is said to be \emph{$C$-symmetric}. A $1$-symmetric basis will be called \emph{isometrically symmetric}. 

The tools used to study spreading sequences could be adapted for symmetric sequences. However, we propose an alternative approach. Specifically, we will do our best to make the most of the results achieved in Section~\ref{sect:Spreading}.

Given a set $\Mt$, $\Pt_{\infty,\infty}(\Mt)$ stands for the sets of all $A\in\Pt_\infty(\Mt)$ such that $\Mt\setminus A$ is infinite. Given another set $\Nt$, we denote by $\Gamma_\infty(\Nt,\Mt)$ the set of all $\psi\in\Gamma(\Nt,\Mt)$ such that $\Mt\setminus\psi(\Nt)\in\Pt_\infty(\Mt)$.

\begin{proposition}\label{prop:sym}
Let $\XB=(\xx_n)_{n=1}^\infty$ be a symmetric sequence in a quasi-Banach space $\XX$.
\begin{enumerate}[label=(\roman*),leftmargin=*,widest=ii]
\item 
For each $\Nt\in\Pt(\NN)$, every $\psi\in\Gamma(\Nt,\NN)$ is a translation relative to $\XB$.
\item $\XB$ is spreading.
\end{enumerate}
\end{proposition}

\begin{proof}
Assume that $\XB$ is not constant and that $\enbrak{\XB}=\XX$. Denote by $\St$ the set consisting of all $\Nt\in\Pt_\infty(\NN)$ such that every map in $\Gamma_\infty(\Nt,\NN)$ is a translation relative to $\XB$. We structure the proof as a series of statements, each building on the previous ones.

\begin{claim}\label{cl:A}
If $\Nt\in\Pt_{\infty,\infty}(\NN)$, then $\Nt\in\St$.
\end{claim}

Indeed, given $\psi\in\Gamma_\infty(\Nt,\NN)$, there is $\pi\in\Pi(\NN)$ such that $\restr{\pi}{\Nt}=\psi$.

\begin{claim}\label{cl:B}
If $\Nt\in\Pt_{\infty,\infty}(\NN)$, then any arrangement of $(\xx_n)_{n\in \Nt}$ is spreading.
\end{claim}

Indeed, since $\Gamma(\Nt,\Nt)\subset\Gamma_\infty(\Nt,\NN)$, the statement follows from Claim~\ref{cl:A}.

\begin{claim}\label{cl:Z}
Let $\Nt\in\Pt_\infty(\NN)$. Then, $\Nt\in\St$ if and only if there is $\varphi\in\Gamma_\infty(\Nt,\NN)$ that is a translation relative to $\XB$.
\end{claim}

Indeed, given $\psi \in\Gamma_\infty(\Nt,\NN)$, $\Mt:=\psi(\Nt)\in \St$ by Claim~\ref{cl:A}. Hence, the map $\pi\in\Upsilon(\Mt,\NN)$ defined by $\pi(\Mt)=\varphi(\Nt)$ is a translation. Since $\psi=\pi^{-1} \circ \varphi$, $\psi$ is a translation relative to $\XB$.

\begin{claim}\label{cl:C}
$\XB$ is linearly independent.
\end{claim}

Indeed, by Lemma~\ref{lem:SNSubSym} and Claim~\ref{cl:B}, $(\xx_n)_{n\in F}$ is linearly independent for all $F\in\Pt_{<\infty}(\NN)$.

\begin{claim}\label{cl:D}  
For each $j$, $k\in\NN$ with $j\not=k$ there is $f^*\in\XX^*$ such that $f^*(\xx_j)=1$, $f^*(\xx_k)=-1$, and $f^*(\xx_n)a=0$ for all $n\in\NN\setminus\{j,k\}$.
\end{claim}

Indeed, if $\psi\in\Pi(\NN)$ is the cycle that swaps the indices $j$ and $k$, then the operator $S\coloneqq\Id_\XX-T_\psi$ satisfies $S(\xx_n)=0$ for all $n\in\NN\setminus\{j,k\}$, $S(\xx_j)=\xx_j-\xx_k$, and $S(\xx_k)=\xx_k-\xx_j$. Hence, there is $f^*\in\XX^*$ such that $S(f)=f^*(f)(\xx_j-\xx_k)$ for all $f\in\XX$. Since $\xx_k\not=\xx_j$ by Claim~\ref{cl:C}, $f^*$ satisfies the desired proporties.

\begin{claim}\label{cl:E}
Given $\Nt\in\Pt_{\infty,\infty}(\NN)$ and $k\in\NN\setminus\Nt$, then $\xx_k\notin \YY\coloneqq\enbrak{\xx_n \colon n\in\Nt}$.
\end{claim}

Indeed, combining Claim~\ref{cl:B} with Corollary~\ref{cor:SSMBounded} yields a bounded sequence $(\yy_n^*)_{n\in \Nt}$ in $\YY^*$ such that $(\xx_n,\yy_n^*)_{n\in \Nt}$ is a biorthogonal system. By Claim~\ref{cl:C} and Claim~\ref{cl:D}, for each $n\in \Nt$, there is $f_n^*\in\XX^*$ such that $\restr{f_n^*}{\YY}=\yy_n^*$ and $f_n^*(\xx_k)=-1$. Suppose that $\xx_k\in\YY$. On the one hand, $\yy_n^*(\xx_k)=-1$ for all $n\in\Nt$. On the one hand, 
\[
\lim_{n\in \Nt} \yy_n^*(\xx_k)=0.
\]
This absurdity proves the claim.

\begin{claim}\label{cl:F}
If $\Nt\in\St$ and $F\in\Pt_{<\infty}(\NN)$, then $F\cup\Nt\in\St$.
\end{claim}

It suffices to see this statement in the case when $F$ is a singleton. Pick $k\in\NN\setminus\Nt$ and $\psi\in \Gamma_\infty(\Nt\cup\{k\},\NN)$. Since $\varphi\coloneqq\restr{\psi}{\Nt}\in \Gamma_\infty(\Nt,\NN)$, it is a translation relative to $\XB$. Besides, $\psi(k)\notin \varphi(\Nt)$, so by Claim~\ref{cl:E} the operator $T_\varphi$ extends to an isomorphism
\[
T\colon \enbrak{\xx_n \colon n\in \Nt\cup\{k\}}\to  \enbrak{\xx_n \colon n\in \varphi(\Nt)\cup\{\psi(k)\}}
\]
such that $T(\xx_k)=\xx_{\psi(k)}$.

\begin{claim}\label{cl:W}
$\Pt_\infty(\NN)\subset\St$.
\end{claim}

Assume by contradiction that there is $\Mt\in\Pt_\infty(\NN)\setminus\St$. Then, by Claim~\ref{cl:F}, $\Mt\setminus F\notin\St$ for all $F\in\Pt_{<\infty}(\Mt)$. Pick an unbounded sequence $(R_n)_{n=1}^\infty$ in $(0,\infty)$. Arrange the elements of $\Mt$ by means of a bijection $\nu\colon\NN\to\Mt$. We recursively construct for each $k\in\NN$ $n_k\in\NN$, $F_k$, $G_k\in \Pt_{<\infty}(\NN)$, $f_k$, $g_k\in\spn(\XB)$ and $A_k\in \Pt_{\infty,\infty}(\NN)$ as follows. Define $F_0=\emptyset$ and choose $A_0\in\Pt_{\infty,\infty}(\NN)$. Pick $j\in\NN$ and assume that if $j\ge 2$ $(n_k, G_k,f_k,g_k,F_k,A_k)_{k=1}^{j-1}$ are constructed. Set 
\[
n_j=j+\max(F_{j-1}), \quad G_j=F_{j-1} \cup\{n_j\}.
\]
Let $\psi_j\in\Gamma(\Mt\setminus \nu(G_j), \NN\setminus A_{j-1})$. By Claim~\ref{cl:Z}, $\psi_j$ is not a translation relative $\XB$, whence
there is
$B_j\in\Pt_{<\infty}(\Mt\setminus \nu(G_j))$ and $a_{j,n}\in\FF$ for each $n\in B_j$  such that
\[
\max\enbrace{ \frac{\norm{g_j}}{\norm{f_j}}, \frac{\norm{f_j}}{\norm{g_j}}}\ge R_j,
\]
where $f_j=\sum_{n\in B_j} a_{j,n}\, \xx_n$ and $g_j=\sum_{n\in B_j} a_{j,n}\, \xx_{\psi_j(n)}$. Set $F_j=F_{j-1}\cup \nu^{-1}(B_j)$ and $A_j=A_{j-1}\cup \psi_j(B_j)$.

Put $\Nt=\cup_{k=1}^\infty B_k$. Since $(B_k)_{k=1}^\infty$ is pairwise disjoint, there is $\psi\colon\Nt\to\NN$ such that $\restr{\psi}{B_k}=\restr{\psi_k}{B_k}$. The sequence $(n_k)_{k=1}^\infty$ is increasing and 
\[
n_k\notin \cup_{j=1}^\infty F_j, \quad k\in\NN,
\]
so it follows that $\Nt\in\Pt_{\infty,\infty}(\NN)$. Furthermore, $A_{0}\cap \psi_k(B_k)=\emptyset$ for all $k\in\NN$, so $\psi\in\Gamma_\infty(\Nt,\NN)$. By Claim~\ref{cl:A}, $\psi$ is a translation relative to $\XB$. Since $T_\psi(f_k)=g_k$ for all $j\in\NN$, we arise to contradiction. Thus, Claim~\ref{cl:W} holds.

We are now ready to complete the proof. Fix $\Mt\in\Pt_{\infty,\infty}(\NN)$. Pick $\Nt\in\Pt_\infty(\NN)$ and $\psi\in\Gamma(\Nt,\NN)$. Choose a bijection $\varphi\colon\psi(\Nt)\to\Mt$. By Claim~\ref{cl:W}, $\varphi$ and $\varphi\circ \psi$ are translations relative to $\XB$. Therefore, $\psi$ is a translation relative to $\XB$.
\end{proof}
   
\begin{corollary}\label{cor:SymBounded}
Let $\XB$ be a symmetric sequence in a quasi-Banach space $\XX$. If $\XB$ is not constant, then it is an M-bounded semi-normalized minimal system of its closed linear span.
\end{corollary}

\begin{proof}
Just combine Proposition~\ref{prop:sym} with Corollary~\ref{cor:SSMBounded}.
\end{proof}

We prove the quantitative version of Proposition~\ref{prop:sym} following the steps of the corresponding one in the unconditional case (see, e.g., \cite{LinTza1977}*{Proposition 3.a.3}).
\begin{proposition}
 Let $C\in[1,\infty)$ and $\XB$ be a $C$-symmetric basis of a quasi-Banach space $\XX$. Then, $\XB$ is $C$-spreading.   
\end{proposition}

\begin{proof}
Pick $\Nt\in\Pt(\NN)$ and $\psi\in\Gamma(\Nt,\NN)$. Pick $f\in \spn(\xx_n \colon n\in\Nt)$. There is $\varphi\in\Pi(\NN)$ such that $\restr{\psi}{\supp(f)}=\restr{\varphi}{\supp(f)}$. Since $T_\psi(f)=L_\varphi(f)$, $\norm{T_\psi(f)} \le C \norm{f}$.
\end{proof}

We go on with the symmetric counterpart of Proposition~\ref{prop:uniform-spread}.

\begin{proposition}\label{prop:SymUniform}
Let $\XB$ be a symmetric basis of a quasi-Banach space $\XX$.  Then,  $\XB$ is $C$-symmetric for some $C\in[1,\infty)$.
\end{proposition}

\begin{proof}
Assume that $\XX$ is a $p$-Banach space, $0<p\le 1$. By Corollary~\ref{cor:SymBounded}, there is $C\in[1,\infty)$ such that
\[
\norm{L}_{\psi}\le C, \quad \psi\in \bigcup_{n=1}^\infty \Gamma\enpar{\enbrace{n},\NN}.
\]

Assume by contradiction that 
\[
\sup\enbrace{ \norm{L_\psi} \colon \psi\in\Pi(\NN)}=\infty.
\]
 Pick an unbounded sequence $(R_k)_{k=1}^\infty$ in $(1,\infty)$. We recursively construct $(f_k,\psi_k)_{k=1}^\infty$ in $\spn(\XB)\times \Pi(\NN)$ as follows. Pick $j\in\NN$  and assume that if $j\ge 2$ $(f_k,\psi_k)_{k=1}^{j-1}$ are constructed. Set 
\[
A_j=\bigcup_{k=1}^{j-1} \supp(f_k), \quad m_j=\abs{A_j}.
\]
Choose $f_j\in\spn(\XB)$ and $\psi_j\in\Pi(\NN)$ such that
\[
\norm{f_j}=m_j, \quad \norm{L_{\psi_j}(f_j)} > C R_j m_j.
\]

Set $g_k=f_k-S_{A_k}(f_k)$ and $B_k=\supp(g_k)$ for all $k\in\NN$. The sequence $(B_k)_{k=1}^\infty$ is pairwise disjoint in $\Pt_{<\infty}(\NN)$, so
there is $\psi\in\Pi(\NN)$ such that $\restr{\psi}{B_k}=\restr{\psi_k}{B_k}$ for all $k\in\NN$. Since
\[
\norm{g_k}^p \le m_k^p + C^p m_k^p, \quad \norm{L_{\psi_k} (g_k)}^p> C^p R_k^p m_k^p-C^p m_k^p,
\]
and $L_{\psi}(g_k)=L_{\psi_k}(g_k)$ for all $k\in\NN$,
\[
\norm{L_{\psi}(g_k)} > \frac{C}{\enpar{1+C^p}^{1/p}} \enpar{R_k^p-1}^{1/p} \norm{g_k}, \quad k\in\NN.
\]
Therefore, $\psi$ is not a shift for $\XB$. This absurdity shows that $(L_\psi)_{\psi\in\Pi(\NN)}$ is uniformly bounded.
\end{proof}

\begin{theorem}\label{thm:IsoSym}
Let $0<p\le 1$ and $\XB$ be a symmetric basis of a locally $p$-convex quasi-Banach space $\XX$. Then, there is an equivalent $p$-norm for $\XX$ so that $L_\psi$ is an isometry for all $\psi\in\Pi(\NN)$.
\end{theorem}

\begin{proof}
Just combine Proposition~\ref{prop:SymUniform} with Lemma~\ref{lem:SG} applied to the group $\Gt_m$.
\end{proof}

Theorem~\ref{thm:IsoSym} complements the study of conditional symmetric bases carried out in \cite{AABCO2024}, where the authors dealt with semi-normalized M-bounded bases of Banach spaces and took for granted that symmetric bases admit isometric renormings. In hindsight, we can now state the following result.

\begin{theorem}[cf.\@ \cite{AABCO2024}*{Corollary 3.8}]\label{thm:AAOB}
Given a symmetric basis $\XB$ of a Banach space $\XX$, the following are equivalent.
\begin{enumerate}[label=(\roman*), leftmargin=*]
    \item $\XB$ is an unconditional basis.
    \item $\XB$ is bidemocratic.
    \item $\XB$ is unconditional for constant coefficients.
    \item $\XB$ is a Markushevich basis.
\end{enumerate}
\end{theorem}

Recall that a complete minimal system $\XB=(\xx_n)_{n=1}^\infty$ is said to be \emph{unconditional for constant coefficients} if there is a constant $C$ such that
\[
\norm{\sum_{n\in A} \varepsilon_n\, \xx_n} \le C \norm{\sum_{n\in A} \varepsilon_n\,\xx_n}
\]
for all $B\in\Pt_{<\infty}(\NN)$, all $A\subset B$, and all
$(\varepsilon_n)_{n\in B}$ in $S_\FF$. Let $(\xx_n^*)_{n=1}^\infty$ be the coordinate functionals of $\XB$. We say that $\XB$ is \emph{bidemocratic} if there is a constant $D$ such that 
\[
\norm{\sum_{n\in A} \varepsilon_n\, \xx_n} \norm{\sum_{n\in B} \delta_n\, \xx_n} \le D \max\enbrace{\abs{A}, \abs{B}}
\]
for all $A$, $B\in\Pt_{<\infty}(\NN)$, and all $(\varepsilon_n)_{n\in A}$, $(\delta_n)_{n\in B}$ in $S_\FF$.

We point out that Theorem~\ref{thm:AAOB} breaks down in the locally non-convex setting. Indeed, the authors of \cite{ABBA2025} constructed a quasi-Banach space $\XX$ with a conditional symmetric Markushevich basis. This space $\XX$ is locally $p$-convex for all $0<p<1$, and its canonical basis is almost greedy, then unconditional for constant coefficients. We remark that this example from  \cite{ABBA2025} does not rule out possible relations between totality and weakened forms of unconditionality, such as quasi-greediness or unconditionality for constant coefficients, in the general framework of quasi-Banach spaces. The summing basis of $c_0$ witnesses that the spreading counterpart of Theorem~\ref {thm:AAOB} does not hold. In contrast, it is unclear to the authors whether a spreading Markushevich basis that fails to be a Schauder basis exists.
\section*{Statements and Declarations}
\subsection*{Conflict of interest} The authors have no competing interests to declare that are relevant to the content of this article.

\subsection*{Data Availability} Data sharing does not apply to this article as no datasets were generated or analysed during the current study.

\end{document}